\documentclass[oneside]{amsart}
\usepackage{enumerate,amssymb}
\usepackage{mathrsfs}
\newtheorem{thm}{Theorem}[section]

\newtheorem{prop}[thm]{Proposition}

\newtheorem{claim}{Claim}
\newtheorem*{thms}{Theorem}
\newtheorem{lem}[thm]{Lemma}
\theoremstyle{definition}
\newtheorem{defn}[thm]{Definition}
\newtheorem{ex}[thm]{Example}
\newtheorem{rem}[thm]{Remark}

\newtheorem{question}[thm]{Question}
\newcommand{\Rset}{\mathbb{R}}
\newcommand{\Nset}{\omega}

\newcommand{\Pset}{\omega^\omega}
\newcommand{\co}{\mathfrak{c}}
\newcommand{\abs}[1]{\lvert#1\rvert}
\newcommand{\seq}[1]{\langle#1\rangle}
\newcommand{\hm}{\mathscr H}
\newcommand{\pack}{\mathscr P}
\newcommand{\leb}{\mathscr L}
\newcommand{\eps}{\varepsilon}
\newcommand{\del}{\delta}
\newcommand{\subs}{\subseteq}
\renewcommand{\leq}{\leqslant}
\renewcommand{\geq}{\geqslant}

\DeclareMathOperator{\hdim}{\dim_{\mathsf{H}}}
\DeclareMathOperator{\pdim}{\dim_{\mathsf{P}}}
\DeclareMathOperator{\diam}{diam}
\DeclareMathOperator{\dist}{dist}
\DeclareMathOperator{\lip}{lip}
\newenvironment{enum}{\begin{enumerate}[\rm(i)]}{\end{enumerate}}
\newenvironment{itemyze}%
  {\begin{list}{\textbullet}{\labelwidth1ex\setlength{\leftmargin}{1.5em}}}%
  {\end{list}}
\newcommand{\si}{$\sigma$\nobreakdash-}
\newcommand{\upto}{{\nearrow}}
\newcommand{\downto}{\downarrow}
\newcommand{\VV}{\mathcal V}
\newcommand{\UU}{\mathcal U}
\newcommand{\NN}{\mathcal{N}}

\newcommand{\FF}{\mathcal{F}}
\newcommand{\EE}{\mathcal{E}}
\newcommand{\JJ}{\mathcal{J}}
\DeclareMathOperator{\add}{\mathbf{add}}
\DeclareMathOperator{\cov}{\mathbf{cov}}
\newcommand{\forinfty}{\forall^\infty}
\newcommand{\exinfty}{\exists^\infty}
\newcommand{\womega}{\widehat\omega}
\begin{document}
\title
[Mapping Analytic sets onto cubes by little Lipschitz functions]
{Mapping Analytic sets onto cubes\\ by little Lipschitz functions}
\author{Jan Mal\'y}
\address{Jan Mal\'y\\
Department of Mathematics\\
Faculty of Science,
 J. E. Purkyn\v e University\\
\v Cesk\'e ml\'ade\v{z}e 8\\
400 96 \'Ust\'{\i} nad Labem\\
Czech Republic}
\email{maly@karlin.mff.cuni.cz}
\author{Ond\v rej Zindulka}
\address
{Ond\v rej Zindulka\\
Department of Mathematics\\
Faculty of Civil Engineering\\
Czech Technical University\\
Th\'akurova 7\\
160 00 Prague 6\\
Czech Republic}
\email{ondrej.zindulka@cvut.cz}
\urladdr{http://mat.fsv.cvut.cz/zindulka}

\begin{abstract}
A mapping $f:X\to Y$ between metric spaces is called \emph{little Lipschitz}
if the quantity
$$
  \lip\!f(x)=\liminf_{r\to0}\frac{\diam f(B(x,r))}{r}
$$
is finite for every $x\in X$.

We prove that if a compact (or, more generally, analytic) metric space
has packing dimension greater than $n$, then $X$ can be mapped
onto an $n$-dimensional cube by a little Lipschitz function.

The result requires two facts that are interesing in their own right.
First, an analytic metric space $X$ contains, for any $\eps>0$,
a compact subset $S$ that embeds into an ultrametric space by a Lipschitz map,
and $\pdim S\geq\pdim X-\eps$.
Second, a little Lipschitz function on a closed subset admits a little
Lipschitz extension.
\end{abstract}

\maketitle

\section{Introduction}

For a mapping $f:X\to Y$ between metric spaces consider the
\emph{lower scaled oscillation} of $f$ at $x$
$$
  \lip\! f(x)=\liminf_{r\to0}\frac{\diam f(B(x,r))}{r}.
$$
Recently there has been a lot of interest in the behavior of functions
at points where $\lip\! f(x)$ is finite and in particular in functions
that have only a few points with $\lip\! f(x)=\infty$.

Differentiability of such functions is studied, e.g., in~\cite{MR2213746,MR3511937}.
The structure of exceptional set  and typical behavior of lower
scaled oscillation is investigated in \cite{BHRZ}
and preservation of measures and dimensions is examined in \cite{HPZZ}.

Keleti, M\'ath\'e and the author of this paper studied in~\cite{MR3159074}
a seemingly totally unrelated question --
what metric spaces can be mapped onto
a cube in an $n$-dimensional Euclidean space by a Lipschitz function.
In particular, they proved that if $X$ is an analytic metric space with Hausdorff
dimension greater than $n$, then $X$ can be mapped onto an $n$-dimensional
cube by a Lipschitz function. The crucial ingredients of this result are
\begin{itemyze}
\item a theorem by Mendel and Naor~\cite{MR3032324} by which every analytic metric
space $X$ contains a Lipschitz copy $S\subs X$ of an ultrametric space
whose Hausdorff dimension is large,
\item a classical theorem that a Lipschitz function on a subset
has a Lipschitz extension over the whole space.
\end{itemyze}

Call a function $f$ on a metric space $X$ \emph{little Lipschitz}
if $\lip\!f(x)$ is finite for each point $x\in X$.
The goal of the present paper is to show that a similar mapping theorem holds
for little Lipschitz functions if Hausdorff dimension is replaced with
the packing dimension. The overall structural pattern of the proof
is very similar to that in~\cite{MR3159074}. We were even able to reiterate
ideas underlying some theorems, in particular the use of the notion of monotone space
introduced in \cite{MR2957686} and \cite{MR2979649}.
However, some key ideas had to be worked out from scratch.
In particular, the counterparts of the two above ingredients that we could use
in~\cite{MR3159074} as black boxes
were unavailable for the context of packing dimension and little Lipschitz functions.
Their formulation and proofs form important parts of the present paper that
are of independent interest.

The first main result of the paper is the following theorem.
Denote by $\pdim$ the packing dimension (precise definition is recalled below).
\begin{thms}
If $X$ is an analytic metric space such that $\pdim X>s$,
then there is a set $S\subs X$ such that $\pdim S>s$ and $S$ is Lipschitz
equivalent to an ultrametric space.
\end{thms}
This theorem is a perfect packing analogy of~\cite[Theorem 1.4]{MR3032324}.
We derive it from~\cite[Theorem 1.5]{MR3032324} in Section~\ref{sec:MN}.

While extending Lipschitz function is rather easy and straightforward,
an analogous theorem for little Lipschitz functions is harder.
We prove it in Section~\ref{sec:maly}.
\begin{thms}
If $F$ is a closed set in a metric space $X$, then
every little Lipschitz function on $F$ has a little Lipschitz extension
over $X$.
\end{thms}

The main theorem of the paper parallels perfectly \cite[Theorem 2.6]{MR3159074}.
It is proved in Section~\ref{sec:main}.
\begin{thms}
Let $X$ be an analytic metric space. If $\pdim X>n$, then there is a
little Lipschitz surjective mapping $f:X\to[0,1]^n$.
\end{thms}

The conclusion fails if the assumption that $X$ is analytic is dropped.
We prove this in section~\ref{sec:ZFC}. We actually show that there are
spaces of large packing dimension that do not map
onto an interval by any continuous function, see Theorem~\ref{add}.
This nicely complements~\cite[Theorem 3.1]{MR3159074} that asserts
that there are spaces of large \emph{Hausdorff} dimension that do not map onto
an interval by a \emph{uniformly} continuous function.

Section~\ref{sec:pack} recalls the notion of packing measure.
Section~\ref{sec:rem} lists a few remarks and open problems.

\medskip

All spaces under consideration are separable metric spaces. Recall that a metric space
is \emph{analytic} if it is a continuous image of a complete metric space
(or, equivalently, of the irrational numbers, or equivalently, a Suslin set in
a complete metric space). A continuous image of an analytic space is analytic.
Every analytic space is separable.

The notion of Lipschitz mapping is well-known. Two metric spaces are
\emph{Lipschitz-equivalent} if there is a Lipschitz bijection with
a Lipschitz inverse.

Some of the common notation includes
$B(x,r)$ for the closed ball centered at $x$, with radius $r$;
$d$ is a generic symbol for a metric;
we write $\diam E$ for the diameter of a set $E$ in a metric space and
$\dist(x,B)$ for the distance from a point $x$ to a set $B$.
Letters $n,m,i,j,k$ are generic symbols for positive integers.
$\Rset$ denotes the real line and $\Rset^n$ the Euclidean space;
$\omega$ stands for the set of natural numbers including zero;
$\Pset$ denotes the set of sequences of natural numbers.
$\abs A$ denotes the cardinality (finite or infinite) of the set $A$.
$X_n\upto X$ means that $\seq{X_n}$ is an increasing sequence of
sets with union $X$.

\section{Packing measures}\label{sec:pack}

In this section we recall the notion of packing measure. There are many definitions;
we adhere to the one from \cite{MR1346667}.

Let $X$ be a metric space and $E\subs X$. Recall that a family of balls
$\{B(x_i,r_i):i\in I\}$ is called a \emph{packing of $E$} if
$x_i\in E$ and $x_j\notin B(x_i,r_i)$ for each $i\neq j$.
If $\del>0$, the packing is called \emph{$\del$-fine} if $r_i\leq\del$ for all
$i$.

Let $s>0$. Define
\begin{align*}
  \pack^s_\del(E)&=\sup\Bigl\{\sum_{i\in I}r_i^s:
  \text{$\{B(x_i,r_i):i\in I\}$ is a $\del$-fine packing of $E$}\Bigr\},\\
  \pack^s_0(E)&=\inf_{\del>0}\pack^s_\del(E).
\end{align*}
The set function $E\mapsto\pack^s_0(E)$ is not an outer measure;
it is subadditive, but not countably subadditive. That is why one more step is required:
The set function
$$
  \pack^s(E)=\inf\sum_n\pack^s_0(E_n)
$$
where the infimum is taken over all countable covers $\{E_n:n\in\Nset\}$ of $E$,
is a countably subadditive Borel regular outer measure whose restriction
to Borel sets is a countably additive Borel measure. It is termed the
\emph{$s$-dimensional packing measure} of $E$.
We refer to the two papers by Edgar~\cite{MR1844397,MR2288081} for a thorough
review of the topic.

\begin{rem}\label{packing}
Our definition of packing coincides with
the definition from \cite{MR1346667}. In literature, packing are often defined
by the requirement that the balls forming it are pairwise disjoint \cite{MR1333890},
or that $r_i+r_j<d(x_i,x_j)$ \cite{MR2288081}.
Our packing is called a \emph{weak packing} in \cite{MR2288081}.
If $\{B(x_i,r_i):i\in I\}$ is a packing according to our definition,
then $\frac12r_i+\frac12r_j>d(x_i,x_j)$ and this in turn implies that
the balls $B(x_i,\frac12 r_i)$ and $B(x_j,\frac 12r_j)$ are disjoint.
None of these implications can be reversed.
\end{rem}

We will utilize the following counterpart of the famous Frostman Lemma for
packing measures:

\begin{lem}\label{FrostmanLemma}
Let $X$ be an analytic metric space and $s>0$. Then $\pack^s(X)>0$
if and only if there is a finite Borel measure $\mu$ on $X$ such that
$\mu(X)>0$ and
\begin{equation}\label{Frostman}
  \forall x\in X\ \exists r_i\downto0 \quad \mu B(x,r_i)\leq r_i^s.
\end{equation}
This measure can be chosen to satisfy $\mu\leq\pack^s$.
\end{lem}
The proof is a simple application of two theorems about packing
measures. The first one is due to Joyce and Preiss~\cite[Theorem 1]{MR1346667}:
\begin{thm}[{\cite{MR1346667}}]\label{JP}
If $X$ is analytic and $\pack^s(X)>0$, then there is a compact set $C\subs X$
such that $0<\pack^s(C)<\infty$.
\end{thm}
The other is a density inequality for packing measures, as it appears
in ~\cite{MR2288081}.
For a finite Borel measure $\mu$ on $X$ and $s>0$, the
\emph{lower density} of $\mu$ at $x\in X$ is
$$
  \Theta_*^s(\mu,x)=
  \liminf_{r\downto0}\dfrac{\mu B(x,r)}{r^s}.
$$
\begin{thm}[{\cite{MR2288081}}]\label{Edgar}
Let $\mu$ be a finite Borel measure in $X$, $E\subs X$ and $s>0$. Then
\begin{equation}\label{Edgar2}
  2^{-s}\pack^s(E)\cdot\inf_{x\in E}\Theta_*^s(\mu,x)\leq
  \mu(E)\leq\pack^s(E)
  \cdot\sup_{x\in E}\Theta_*^s(\mu,x).
\end{equation}
The right-hand side inequality holds provided
the rightmost product is not $0\cdot\infty$.
\end{thm}
\begin{proof}
The first inequality follows from \cite[Theorem 5.9]{MR2288081} and
Remark~\ref{packing}.
The second inequality is \cite[Theorem 5.29]{MR2288081}.
\end{proof}

\begin{proof}[Proof of Lemma~\ref{FrostmanLemma}]
Let $C$ be the compact set of Theorem~\ref{JP}.
Define a Borel measure $\nu$ on $X$ by $\nu(E)=\pack^s(E\cap C)$.
By the left-hand side inequality~\eqref{Edgar2} we have
\begin{equation}\label{fr3}
  \inf_{x\in E}\Theta_*^s(\nu,x)\leq 2^s\frac{\nu(E)}{\pack^s(E)}\leq 2^s
\end{equation}
for every Borel set $E$ of positive measure $\nu$.
Let $A=\{x\in X:\Theta_*^s(\nu,x)>2^s\}$. If $\nu(A)>0$, then there is $\eps>0$
such that the set $E=\{x\in X:\Theta_*^s(\nu,x)>2^s+\eps\}$ also satisfies $\nu(E)>0$,
which contradicts \eqref{fr3}.
It follows that $\nu(A)=0$, i.e.,
$\Theta_*^s(\nu,x)\leq2^s$ $\nu$-a.e. Choose a closed set $F$ such that
$\nu(F)>0$ and $\Theta_*^s(\nu,x)\leq2^s$ for all $x\in F$.
Now define $\mu(E)=2^{-s-1}\nu(E\cap F)$. Clearly $\Theta_*^s(\mu,x)\leq\frac12$
for $x\in F$ and since $F$ is closed, $\Theta_*^s(\mu,x)=0$ for $x\notin F$.
The measure $\mu$ thus obviously satisfies~\eqref{Frostman}.

The other direction is trivial: if $\mu$ satisfies~\eqref{Frostman}, then clearly
$\Theta_*^s(\mu,x)\leq 1$ for all $x\in X$ and $\pack^s(X)\geq\mu(X)>0$ follows from
the right-hand side inequality~\eqref{Edgar2}.
\end{proof}

We will also need a simple covering lemma:
\begin{lem}\label{vitali}
Let $E\subs X$ and let $\{B(x,r_x):x\in E\}$ be a collection of
balls such that $\sup_{x\in E}r_x<\infty$.
For each $\alpha>1$ one can extract a countable packing
$\{B(x,r_x):x\in D\}$ such that
$\{B(x,\alpha r_x):x\in D\}$ covers $E$.
\end{lem}
\begin{proof}
Assume without loss of generality that $r_x<1$ for each $x\in E$.
For $n\in\Nset$ define inductively
\begin{align*}
  A_n &=\{x\in E:\alpha^{-n+1}>r_x\geq\alpha^{-n}\},
  \\
  B_n &=A_n\setminus\bigcup\{B(x,r_x):x\in\bigcup\nolimits_{i<n}\VV_i\},
\end{align*}
and let $\VV_n\subs\{B(x,r_x):x\in B_n\}$ be a packing
that is maximal among all packings extracted from
$\{B(x,r_x):x\in B_n\}$. Its existence verifies by a
standard Zorn Lemma argument.
Eventually put $\VV=\bigcup_{n\in\Nset}\VV_n$.
Routine verification proves that $\VV$ is the required packing.
%
%
\end{proof}

\section{Large ultrametric subset}\label{sec:MN}

In this section we prove a counterpart of Mendel and Naor~\cite[Theorem 1.4]{MR3032324}
mentioned in the introduction.

Recall that a metric space $(X,d)$ is \emph{ultrametric} if
the triangle inequality reads $d(x,z)\leq\max\{d(x,y),d(y,z)\}$.
Let us call a metric space $S$ \emph{Lipschitz-ultrametric} if
there is a Lipschitz bijection $f:S\to U$ onto an ultrametric space $U$.

Let $E$ be a set in a metric space. The packing dimension of $E$
is denoted and defined by
$$
 \pdim E=\inf\{s>0:\pack^s(E)=0\}=\sup\{s>0:\pack^s(E)=\infty\}.
$$
Properties of packing dimension, including various equivalent definitions,
are well-known. We refer to~\cite{MR1333890}
or~\cite{MR2118797}. We point out that packing dimension of any set is
greater than or equal to its Hausdorff dimension.
\begin{thm}\label{pack}
Let $X$ be an analytic metric space. For each $\del>0$ there is a
Lipschitz-ultrametric compact set $S\subs X$ such that
$\pdim S\geq\pdim X-\del$.
\end{thm}
The proof is an easy application of~\cite[Theorem 1.5]{MR3032324}:
\begin{thm}[\cite{MR3032324}]\label{mendel}
Let $X$ be an analytic metric space and $\mu$ a finite Borel measure on $X$.
For each $\eps>0$ there is a constant $c_\eps$ and
a compact set $S\subs X$ such that
$S$ is Lipschitz-equivalent to an ultrametric space and
if $\{B(x_i,r_i):i\in I\}$ is a cover of $S$, then
\begin{equation}\label{eq:mendel}
  \sum_{i\in I}\bigl(\mu B(x_i,c_\eps r_i)\bigr)^{1-\eps}\geq(\mu X)^{1-\eps}.
\end{equation}
\end{thm}

\begin{proof}[Proof of Theorem~\ref{pack}]
Let $t=\pdim X-\del$. Choose $p$ and then $\eps>0$
such that $t<(1-\eps)p<p<\pdim X$.

Clearly $\pack^p(X)>0$. By Lemma~\ref{FrostmanLemma} there is a finite Borel measure
$\mu$ on $E$ such that $\mu\leq\pack^p$ and such that
\begin{equation}\label{frostman}
  \forall x\in X\ \exists r_n\searrow0 \quad \mu B(x,r_n)\leq r_n^p.
\end{equation}
Let $S\subs E$ and $c_\eps$ be the set and constant of the Mendel-Naor
Theorem~\ref{mendel}.
Put
\begin{equation}\label{eta}
  \eta=\left(\frac{\mu X}{(2c_\eps)^p}\right)^{1-\eps},
  \qquad
  s=(1-\eps)p.
\end{equation}
Therefore
\begin{equation}\label{eta2}
  \eta(2c_\eps)^{p(1-\eps)}=\eta(2c_\eps)^s<(\mu X)^{1-\eps}.
\end{equation}
We claim that $\pack^t(S)=\infty$. Aiming towards contradiction
suppose that it is not the case. Then there is a sequence of sets
$S_n\upto S$ such that $\pack^t_0(S_n)<\infty$ for all $n$;
this follows from the subadditivity of $\pack^t_0$.
Since $s>t$, we have $\pack^s_0(S_n)=0$. Therefore there is a sequence
$\del_n\downto0$ such that
\begin{equation}\label{eq1}
  \pack^s_{\del_n}(S_n)<\eta\cdot 2^{-n-1}.
\end{equation}
Let
$$
  \VV=\{B(x,r):x\in S,\ x\in S_n\setminus S_{n-1}\Rightarrow r<\del_n,\
               \mu B(x,c_\eps\cdot 2r)\leq(c_\eps\cdot2r)^p\}.
$$
By \eqref{frostman} and Lemma~\ref{vitali}, $\VV$ contains a
subfamily $\{B(x_i,r_i):i\in I\}\subs\VV$ such that
\begin{enum}
\item $\{B(x_i,r_i):i\in I\}$ is a packing,
\item $\{B(x_i,2r_i):i\in I\}$ is a cover of $S$.
\end{enum}
By (ii) and \eqref{eq:mendel} we have
\begin{equation}\label{eq2}
\sum_{i\in I}\bigl(\mu B(x_i,c_\eps 2r_i)\bigr)^{1-\eps}\geq(\mu X)^{1-\eps}.
\end{equation}
On the other hand, we have, for each $n$, by the definition of $\VV$, (i),
\eqref{eta2} and \eqref{eq1}
\begin{multline*}
  \sum_{x_i\in S_n\setminus S_{n-1}}\mu B(x_i,c_\eps 2r_i)^{1-\eps}
  \leq \sum_{x_i\in S_n\setminus S_{n-1}}((c_\eps 2r_i)^p)^{1-\eps}\\
  =(2c_\eps)^s \sum_{x_i\in S_n\setminus S_{n-1}}r_i^s
  \leq(2c_\eps)^s\pack^s_{\del_n}(S_n)\\
  <(2c_\eps)^s\eta\cdot 2^{-n-1}
  <2^{-n-1}(\mu X)^{1-\eps}.
\end{multline*}
Summing up over $n$ we get
$$
  \sum_{i\in I}\mu B(x_i,c_\eps 2r_i)^{1-\eps}
  <(\mu X)^{1-\eps}\sum_{n\in\Nset}2^{-n-1}=(\mu X)^{1-\eps},
$$
a contradiction with \eqref{eq2}. We proved that $\pack^t(S)=\infty$.
Therefore $\pdim S\geq t=\pdim X-\del$, as required.
\end{proof}

\section{Extending Little Lipschitz Function}\label{sec:maly}

In this section we prove that a function on a closed
set with finite lower scaled oscillation
has an extension with the same property over the whole space.

First we recall the notions in consideration.
For a mapping $f:X\to Y$ between metric spaces and $x\in X$, $r>0$ define
the \emph{oscillation} of $f$ on the ball $B(x,r)$
$$
  \omega_f(x,r)=\diam f(B(x,r)).
$$
The \emph{lower scaled oscillation} function is defined by
$$
  \lip\! f(x)=\liminf_{r\to0}\dfrac{\omega_f(x,r)}{r}.
$$
Note that some (e.g., \cite{MR2213746}) define the lower scaled oscillation
from the version of $\omega_f$
given by $\widehat\omega_f(x,r)=\sup_{y\in B(x,r)} d(f(y),f(x))$.
It is clear though that $\widehat\omega_f(x,r)\leq\omega_f(x,r)\leq2\widehat\omega(x,r)$
and thus the two lower scaled oscillation functions differ at most by a factor of $2$.
We will use $\widehat\omega$ in the proof of Theorem~\ref{maly}.

\begin{defn}
\begin{itemyze}
\item $f$ is \emph{little Lipschitz} if
$\lip\! f(x)<\infty$ for all $x\in X$,
\item $f$ is \emph{lower Lipschitz} if there is $L$ such that
$\lip\! f(x)\leq L$ for all $x\in X$.
\end{itemyze}
\end{defn}
Let us point out a trivial but very important fact:
\emph{Every little Lipschitz mapping is continuous.}

We now present the extension theorem.
\begin{thm}\label{maly}
Let $X$ be a metric space and $F\subs X$ a closed subset.
Then every little Lipschitz function $f:F\to\Rset$ has an extension
$f^*:X\to\Rset$ that is little Lipschitz on $F$ and locally Lipschitz
on $X\setminus F$.
In particular, $f^*$ is little Lipschitz on $X$.
\end{thm}
\begin{proof}
First suppose $f$ is bounded.
We may clearly assume that $0\leq f(x)\leq1$ for all $x\in F$.
For any $x\in F$ and $r>0$ set
$$
  s_{x,r}=\sup f(B(x,3r)\cap F).
$$
Let $y\in X$. Define
$$
  f_{x,r}(y)=
  \begin{cases}
    s_{x,r}& \text{if $y\in B(x,r)$},\\
    s_{x,r}+\dfrac{d(x,y)-r}r& \text{if $y\in X\setminus B(x,r)$}.
  \end{cases}
$$
Each of the functions $f_{x,r}$ is clearly $\frac1r$-Lipschitz.
The extension is defined by
$$
  f^*(y)=\min\{\inf\{f_{x,r}(y): x\in F,\;r>0\},1\}.
$$
We prove that it has the required properties.
\begin{claim}
$f^*(y)=f(y)$ for all $y\in F$.
\end{claim}
\begin{proof}
Fix $y\in F$ and let $x\in F$, $r>0$.

If $d(y,x)\leq2r$, then $f_{x,r}(y)\geq s_{x,r}\geq f(y)$.
If $d(y,x)\geq2r$, then $f_{x,r}(y)\geq 1$. In any case,
$f^*(y)\geq f(y)$.

On the other hand, for all $r>0$ we have
$$
  f^*(y)\leq f_{y,r}(y)=s_{y,r}\leq f(y)+\womega_f(y,3r)
$$
and $f^*(y)\leq f(y)$ obtains by letting $r\to0$.
\end{proof}

\begin{claim}
$f^*$ is little Lipschitz on $F$.
\end{claim}
\begin{proof}
Let $x\in F$. Let $r>0$ and $y\in B(x,r)$. We will show that
\begin{equation}\label{maly222}
  \abs{f^*(y)-f(x)}\leq \womega_f(x,3r)
\end{equation}
which is enough. The following estimate is easy:
\begin{equation}\label{maly333}
  f^*(y)\leq f_{x,r}(y)=s_{x,r}\leq f(x)+\womega_f(x,3r).
\end{equation}
We also need a lower estimate of $f^*(y)$. To that end we are estimating
$f_{z,\rho}(y)$ for each $z\in F$ and $\rho>0$.
We consider three cases:
\begin{itemyze}
\item $d(x,z)\leq 3\rho$. In this case,
$f_{z,\rho}(y)\geq s_{z,\rho}\geq f(x)$.
\item $d(x,z)\leq 3r$. In this case,
$
f_{z,\rho}(y)\geq s_{z,\rho}\geq f(z)\geq f(x)-\womega_f(x,3r)
$
\item $d(x,z)\geq \max\{3\rho,3r\}$. In this case, $d(y,z)\geq2\rho$
and thus $f_{z,\rho}(y)\geq1$.
\end{itemyze}
In any case, $f_{z,\rho}(y)\geq f(x)-\womega(x,3r)$.
Therefore $f^*(y)\geq f(x)-\womega_f(x,3r)$, which, together with~\eqref{maly333},
yields \eqref{maly222}.
\end{proof}
For $\eps>0$ write $F^\eps=\{y\in X:\dist(y,F)\leq\eps\}$.
\begin{claim}
For each $\eps>0$, $f^*$ is $\frac2\eps$-Lipschitz on $X\setminus F^\eps$.
\end{claim}
\begin{proof}
If $x\in F^\eps$, $z\in F$ and $r\leq\frac\eps2$, then $f_{z,r}(x)\geq1$.
Therefore
$$
  f^*(x)=\min\{\inf\{f_{z,r}(x):x\in F, r>\tfrac12\eps\},1\},
$$
i.e., $f^*\restriction (X\setminus F^\eps)$ is an infimum of a family of
$\frac2\eps$-Lipschitz functions and is consequently also $\frac2\eps$-Lipschitz.
\end{proof}
Claims 1--3 obviously prove the theorem for the case of bounded $f$.
If $f$ is unbounded, we can use the standard trick: first replace $f$ with
$g=\arctan f$. Then we extend $g$ to $g^*$ with values in $[-\pi/2,\pi/2]$ and set
$$
  f^*(x)=\tan(g^*(x)e^{-\dist(x,F)}).
$$
The correction by $e^{-\dist (x,F)}$ has the effect that the resulting function
is finite-valued.
\end{proof}

The theorem is quite sharp:
the assumption that $F$ is closed cannot be abandoned, which sharply contrasts
the extension of Lipschitz maps. Also, a lower Lipschitz map does not have
to have a lower Lipschitz extension, not even in the case of compact $F$.
See Section~\ref{sec:rem} for details.

\section{Mapping Analytic Space Onto a Cube}\label{sec:main}

In this section we prove the main result of the paper: if $X$ is analytic
and $\pdim X>n$, then $X$ maps onto $[0,1]^n$ by a little Lipschitz function.

We need the notions of little and lower H\"older mapping.
For a mapping $f:X\to Y$ between metric spaces and $x\in X$ and $\beta>0$
define
$$
  \lip_\beta\! f(x)=\liminf_{r\to0}\dfrac{\omega_f(x,r)}{r^\beta}.
$$

\begin{defn}
\begin{itemyze}
\item $f$ is \emph{little $\beta$-H\"older} if
$\lip_\beta\! f(x)<\infty$ for all $x\in X$,
\item $f$ is \emph{lower $\beta$-H\"older} if there is $L$ such that
$\lip_\beta\! f(x)\leq L$ for all $x\in X$.
\end{itemyze}
\end{defn}
Note that, just like little Lipschitz ones, little H\"older maps are continuous.

We will utilize the notion of a monotone metric space introduced in~\cite{MR2957686}
and \cite{MR2979649}. Recall that by one of the equivalent definitions
(see~\cite{MR2979649}), a metric space $X$ is \emph{monotone}
if there is a linear order $<$ on $X$ and a constant $c$ such that
\begin{equation}\label{eq:mono}
  [x,y]\subs B(x,c\cdot d(x,y))
\end{equation}
whenever $x<y$. Here $[x,y]$ denotes the closed interval
with respect to the underlying order.

\begin{thm}\label{main}
Let $S$ be compact monotone metric space. If $\pack^s(S)>0$, then
there is a surjective lower $s$-H\"older function $g:S\to[0,1]$.
\end{thm}
\begin{proof}
By Theorem~\ref{FrostmanLemma}
there is a finite Borel measure $\mu$ on $S$ such that
\begin{equation}\label{frostman2}
  \forall x\in S\ \exists r_{x,n}\downto0 \quad \mu B(x,r_n)\leq r_{x,n}^s.
\end{equation}
Let $<$ be the linear order and $c$ the constant satisfying for all $x<y$
$$
  [x,y]\subs B(x,c\cdot d(x,y)).
$$
Define $g(x)=\mu\{z\in S:z\leq x\}$.

Fix $x\in S$ and suppose that $d(x,y)\leq r_{x,n}/c$.
If $x<y$, then by \eqref{eq:mono} and \eqref{frostman2}
$$
  g(y)-g(x)\leq\mu[x,y]\leq\mu B(x,c\cdot r_{x,n}/c)=
  \mu B(x,r_{x,n})\leq r_{x,n}^s.
$$
If $y<x$, then by the same argument $g(x)-g(y)\leq r_{x,n}^s$.
Overall, $\omega_g(x,r_{x,n}/c)\leq 2r_{x,n}^s$ and thus
$\lip_s g(x)\leq 2c^s$, so $g$ is lower $s$-H\"older with constant
$2c^s$.

Routine calculation shows that $g$ is not constant and that $g(S)$ is an interval.
(The details are worked out in the proof of \cite[Theorem 2.1]{MR3159074}.)
Thus $g$ is the required mapping.
\end{proof}

\begin{thm}\label{main2}
Let $X$ be an analytic metric space. If $\pdim X>n$, then there is a
little Lipschitz surjective mapping $f:X\to[0,1]^n$.
\end{thm}
\begin{proof}
By the assumption and Theorem~\ref{pack} there is a
Lipschitz-ultrametric compact set $S$ such that $\pdim S>n$. In particular,
$\pack^n(S)>0$.
As proved in~\cite{MR2979649} and also in~\cite[Lemma 2.3]{MR3159074}, every
compact ultrametric space is monotone, and since monotonicity is
a bi-Lipschitz invariant, the set $S$ is a monotone metric space.
We may thus apply the above theorem to get a function
$g:S\to[0,1]$ that is onto and lower $n$-H\"older.
It is well known (see, e.g., \cite[Theorem 4.55]{MR560132}) that
there is a Peano curve $\pi:[0,1]\to[0,1]^n$ onto $[0,1]^n$ that is $\frac1n$-H\"older.
The mapping $f=\pi\circ g$, being a composition of a
lower $n$-H\"older and $\frac1n$-H\"older mappings, is lower Lipschitz
and it is clearly onto $[0,1]^n$.

We now extend $f$ into a little Lipschitz function $f^*:X\to[0,1]^n$ as follows:
Break down the function $f$ into coordinate functions. It should be clear from the
proof of Theorem~\ref{maly} that we may extend each of the coordinate functions
in such a way that their aggregate is an extension of $f$ that
is still little Lipschitz.
\end{proof}

\section{Large metric spaces that do not map onto an interval}\label{sec:ZFC}

In this section we prove a theorem that exhibits that the conclusion of
Theorem~\ref{main2} indeed fails if we drop the assumption that $X$
is analytic: there are spaces of large packing
dimension that do not map onto $[0,1]$ by any little Lipschitz
(actually any continuous) function. As pointed out in the introduction,
this theorem complements \cite[Theorem 3.1]{MR3159074}.

\begin{thm}\label{add}
There exist separable metric spaces with arbitrarily large packing dimension
that cannot be mapped onto an interval by a continuous function.
\end{thm}

We first prove the theorem, just like in \cite{MR3159074}, under a
set-theoretic assumption about the \si ideal $\NN$ consisting of
the subsets of $\Rset$ of Lebesgue measure zero. Recall that
given an ideal $\JJ$ of sets in $\Rset$, the covering number $\cov\JJ$ is
defined as the minimal size of a family $\FF\subs\JJ$ that covers $\Rset$.
In the following theorem we consider $\cov\NN$.
Denote by $\co$ the cardinal of continuum. It is well-known that
$\omega<\cov\NN\leq\co$ and that the value of
$\cov\NN$ cannot be determined from the usual axioms of set theory.
We will use $\cov\NN=\co$, which  means that less than continuum many
null sets does not cover the line.
Note that $\cov\NN=\co$ is yielded, e.g., by the Continuum Hypothesis.
The interested reader is referred to~\cite{MR1350295}.
\begin{thm}\label{add1}
If $\cov\NN=\co$, then for any $n$ there is a set $A\subs\Rset^n$ that
is not Lebesgue null and yet
does not map onto an interval by a continuous function.
\end{thm}
From now on, $\leb^n$ denotes the Lebesgue measure in $\Rset$.
We will need a well-known fact, see, e.g., \cite[522Va]{fremlin5}
for the proof.
\begin{lem}\label{ox}
Let $B\subs\Rset^n$ be a Borel set and $\FF$ a family of Lebesgue null sets
in $\Rset^n$.
If $\leb^n(B)>0$ and $\abs\FF<\cov\NN$, then $\FF$ does not cover $B$.
\end{lem}

\begin{proof}[Proof of Theorem~\ref{add1}]
We elaborate the proof of~\cite[Theorem 3.1]{MR3159074}.
Let $\{N_\alpha:\alpha<\co\}$ enumerate all Lebesgue null Borel sets in $\Rset^n$.
Let $\{(G_\alpha,f_\alpha):\alpha<\co\}$ list all pairs $(G,f)$ such that
$G\subs\Rset^n$ is a $G_\del$-set that is not Lebesgue null and $f:G\to\Rset$ is
a continuous function.

For each $\alpha<\co$ construct recursively points $x_\alpha\in\Rset^n$
and $y_\alpha\in[0,1]$ subject to the following conditions:
\begin{enum}
  \item $x_\alpha\in G_\alpha\setminus N_\alpha$,
  \item $x_\alpha\notin\bigcup_{\beta<\alpha}f_\beta^{-1}(y_\beta)$,
  \item $y_\alpha\notin f_\alpha(\{x_\beta:\beta\leq\alpha\})$,
  \item $f_\alpha^{-1}(y_\alpha)$ is Lebesgue null.
\end{enum}
Suppose that at stage $\alpha$ the conditions are met by smaller indices.
Since $\cov\EE=\co$, we may, with the aid of (iv),
apply Lemma~\ref{ox} with $B=G_\alpha\setminus N_\alpha$ and
$\FF=\{f_\beta^{-1}(y_\beta):\beta<\alpha\}$ to conclude that
there is $x_\alpha\in G_\alpha$ so that (i) and (ii) hold.

Since $\abs{f_\alpha(\{x_\beta:\beta\leq\alpha\})}<\co$, there is
$t\notin f_\alpha(\{x_\beta:\beta\leq\alpha\})$
such that $f_\alpha^{-1}(t)$ is Lebesgue null, because otherwise
the family $\{f_\alpha^{-1}(s):s\notin f_\alpha(\{x_\beta:\beta\leq\alpha\})$
would be an uncountable family of pairwise disjoint $G_\del$-sets of positive measure.
Thus letting $y_\alpha=t$, (ii) and (iv) hold.

Let $A=\{x_\alpha:\alpha<\co\}$.
Clearly $A$ is not null by (i). Suppose that there is a continuous function
$f:A\to\Rset$ such that $f(A)$ contains an interval. Mutatis mutandis we may
suppose that $[0,1]\subs f(A)$.
We now employ an extension theorem due to Kuratowski
(see~\cite[(3.8)]{MR1321597} or~\cite[4.3.20]{MR1039321}): there is a continuous extension $f^*$ of $f$
to a $G_\del$-set $G\supseteq A$. Since $A$ is not null, neither is $G$, thus
the pair $(G,f^*)$ is listed, i.e., $(G,f^*)=(G_\alpha,f_\alpha)$ for some $\alpha$.
Now $f_\alpha(x_\beta)\neq y_\alpha$ for $\beta\leq\alpha$ by (iii)
and for $\beta>\alpha$ by (ii). It follows that $y_\alpha\notin f_\alpha(A)$.
Therefore $[0,1]\nsubseteq f_\alpha(A)=f(A)$.
\end{proof}

We now show a construction of a set with infinite packing dimension that
does not map onto an interval by any function. This construction requires
an assumption that is obviously weaker than $\cov\NN<\co$ and
it is also weaker than a condition that was used in
\cite[Theorem 3.8]{MR3159074} for a similar construction.
The additivity of $\add\NN$ is defined as the minimal cardinality
of a family $\FF\subs\NN$ such that $\bigcup\FF\notin\NN$.
Similar notes as about $\cov\NN$ are in place:
$\omega<\add\NN\leq\cov\NN\leq\co$; the value of
$\add\NN$ cannot be determined from the usual axioms of set theory;
$\add\NN=\co$ is yielded, e.g., by the Continuum Hypothesis.
We also note that this cardinal invariant is the smallest one in the so
called \emph{Cicho\'n's Diagram}.
In particular, $\add\NN<\co$ is relatively consistent.
We refer to~\cite{MR1350295} for details.

Recall that $\Pset$ is the set of sequences of
natural numbers together with the product topology. The
topology is metrizable. One of the common metrics
that induce the topology of $\Pset$ is the following \emph{least difference metric}:
for distinct $x,y\in\Pset$ denote by $\abs{x\wedge y}$ the length of the
common initial segment of $x$ and $y$.
Thus $\abs{x\wedge y}=\min\{n\in\Nset:x(n)\neq y(n)\}$.
Define $d(x,y)=2^{-\abs{x\wedge y}}$.
From now on we suppose that $\Pset$ is equipped with this metric.
The family of finite subsets of $\Nset$ is denoted by $[\Nset]^{<\Nset}$. The
quantifiers $\forinfty n$ and $\exinfty n$ read ``for all $n$ except finitely many''
and ``for infinitely many $n$'', respectively.
\begin{thm}\label{add2}
There is a set $E\subs\Pset$ such that $\abs{E}=\add\NN$ and $\pdim E=\infty$.

Consequently, if $\add\NN<\co$ then $E$ does not map onto an interval by any function.
\end{thm}
\begin{lem}\label{lem333}
Let $X$ be a metric space, $E\subs X$ and $s>0$. If $\pack^s(E)=0$,
then there is a sequence
$\seq{I_n}$ of subsets of $E$ such that $\abs{I_n}\leq 2^{ns}$ and
\begin{equation}\label{eq333}
  \forall x\in E\ \forall^\infty n\in\Nset\quad \dist(x,I_n)\leq 2^{-n}.
\end{equation}
\end{lem}
\begin{proof}
Since $\pack^s(E)=0$, there is a sequence $E_n\upto E$ such that
$\pack_{2^{-n}}^s(E_n)<1$. Let $I_n\subs E_n$ be a maximal $2^{-n}$-separated set
(i.e., $d(x,y)>2^{-n}$ for all distinct $x,y\in I_n$).
The family of balls $\UU_n=\{B(x,2^{-n}):x\in I_n\}$
is clearly a $2^{-n}$-fine packing of $E_n$.
Consequently $\abs{I_n}\cdot 2^{-ns}\leq\pack_{2^{-n}}^s(E_n)<1$ and it follows that
$\abs{I_n}<2^{ns}$.

Since $\UU_n$ is maximal, it is also a cover of $E_n$. Together with $E_n\upto E$
this yields~\eqref{eq333}.
\end{proof}
\begin{proof}[Proof of Theorem~\ref{add2}]
We will need a deep result from the set theory of reals.
Call a function
$S:\Nset\to[\Nset]^{<\Nset}$ \emph{slalom} if
$\forinfty n\in\Nset$ $\abs{S(n)}\leq 2^{n^2}$.
Bartoszynski~\cite{MR719666} (or see \cite[Theorem 2.3.9]{MR1350295})
and Fremlin~\cite[522M]{fremlin5} (Fremlin calls slalom a \emph{localization relation})
prove that there is a set $E\subs\Pset$ such that $\abs E=\add\NN$
and for every slalom $S$ there is $x\in E$ such that
\begin{equation}\label{eq444}
  \exinfty n\in\Nset\ x(n)\notin S(n).
\end{equation}
We claim that $\pdim E=\infty$.

Suppose the contrary: there is $s>0$ such that $\pack^s(E)=0$.
Let $\seq{I_n}$ be the sequence from Lemma~\ref{lem333} and define
$S(n)=I_{n+1}$. Then $\forinfty n\ \abs{S(n)}\leq 2^{(n+1)s}\leq 2^{n^2}$,
so $S$ is a slalom. Therefore there is $x\in E$ such that \eqref{eq444} holds.
On the other hand, \eqref{eq333} holds, hence
$\forinfty n\ \exists y_n\in I_{n+1}\ d(y_n,x)\leq 2^{-n-1}$,
which in turn means that $\abs{x\wedge y}\geq n+1$ and in particular
$x(n)=y_n(n)$. Since $y_n\in I_{n+1}=S(n)$, we have $\forinfty n\ x(n)\in S(n)$,
which contradicts \eqref{eq444}.
We proved that $\pack^s(E)>0$ for all $s>0$. Hence $\pdim E=\infty$, as required.
\end{proof}

It is clear that Theorems~\ref{add1} and \ref{add2} yield Theorem~\ref{add}:
If $\cov\NN=\co$ holds, then Theorem~\ref{add1} applies.
If $\cov\NN=\co$ fails, then $\add\NN<\co$ holds and thus Theorem~\ref{add2}
applies.

\section{Remarks}\label{sec:rem}

\subsection*{Generalization of the extension theorem}

An inspection of the proof of Theorem~\ref{maly} reveals that we actually
proved a more general statement. Recall that
$\widehat\omega_f(x,r)=\sup_{y\in B(x,r)} d(f(y),f(x))$.
\begin{thm}
Let $X$ be a metric space and $F\subs X$ a closed subset.
Then every continuous function $f:F\to[0,1]$ has an extension
$f^*:X\to[0,1]$ such that $\widehat\omega_{f^*}(x,r)\leq\widehat\omega_f(x,3r)$
for all $x\in F$ and $f^*$ is locally Lipschitz on $X\setminus F$.
\end{thm}

\subsection*{Sharpness of the extension theorem}

In Theorem~\ref{maly}, the assumption that $F$ is closed cannot be
abandoned. The following theorem comes from~\cite[3.17]{BHRZ}:

\emph{For any perfect nowhere dense set $E\subs[0,1]$ there is a continuous function
$f:[0,1]\to[0,1]$ that is constant on each connected component of $G=[0,1]\setminus E$
and yet $\lip\!f(x)=\infty$ for each $x\in E$.}

It follows that $f\upharpoonright G$ is lower Lipschitz. Since $G$ is dense,
the only continuous extension of $f\upharpoonright G$ over $[0,1]$ is $f$.
Thus $f$ has no little Lipschitz extension.
\medskip

We may also ask if the extension can be constructed to be lower Lipschitz
if $f$ is lower Lipschitz. The answer is negative.
\begin{ex}
There is a lower Lipschitz function $f:C\to[0,1]$ defined
on a compact set $C\subs[0,1]$ that has no lower Lipschitz extension
$f^*:[0,1]\to[0,1]$.
\end{ex}
Such a function is constructed
in~\cite{HPZZ}. It is lower Lipschitz and $\hdim C<\hdim f(C)$.
If $X$ had a lower Lipschitz extension $f^*:[0,1]\to[0,1]$, then it would be
Lipschitz; this follows from \cite{HPZZ} where it is shown that a lower Lipschitz
function on an interval is Lipschitz. But if it were so, then we would have,
as for any Lipschitz function, $\hdim f(C)\leq\hdim C$.

%

\subsection*{Necessity vs.~sufficiency of $\pack^n(X)>0$}

It is easy to show that $\pack^n(S)>0$ is a necessary condition for
Theorem~\ref{main2}:
\begin{prop}
Let $X$ be an analytic space. If $X$ maps by a little Lipschitz function
onto $[0,1]^n$, then $\pack^n(X)>0$.
\end{prop}
\begin{proof}
The following integral inequality is proved in~\cite{HPZZ}.
Let $s>0$. Denote by $\hm^s$ the $s$-dimensional Hausdorff measure.
\emph{If $f:X\to Y$ is little Lipschitz, then
$\hm^s(f(X))\leq\int_X(\lip_s f)^s\,\mathrm{d}\pack^s$.}

Suppose that $\pack^n(X)=0$ and that $f:X\to[0,1]^n$ is little Lipschitz.
Let $s=n$ in the inequality to get
$\hm^n(f(X))\leq\int_X(\lip_s f)^n\,\mathrm{d}\pack^n=0$. It follows that
$f(X)$ is a Lebesgue null set.
\end{proof}

There is a very complicated and obscure example~\cite{MR0154965}
(that was reworked and clarified by Keleti~\cite{MR1313694}) of a compact set $C$
in the plane that has positive $1$-dimensional Hausdorff measure
and yet $C$ does not map onto an interval by any Lipschitz function.
It is not clear though if there is a similar example with $\pack^1(C)>0$
and yet $C$ does not map onto an interval by any little Lipschitz function.
In other words, we do not know if $\pack^n(C)>0$ is sufficient
in Theorem~\ref{main2}, even in the case $n=1$.
\begin{question}
Is there $n$ and a compact space $X$ such that $\pack^n(X)>0$
and $X$ does not map onto $[0,1]^n$ by a little Lipschitz function?
\end{question}

%
\bibliographystyle{amsplain}

\begin{thebibliography}{10}

\bibitem{MR2213746}
Zolt\'an~M. Balogh and Marianna Cs\"ornyei, \emph{Scaled-oscillation and
  regularity}, Proc. Amer. Math. Soc. \textbf{134} (2006), no.~9, 2667--2675.
  \MR{2213746}

\bibitem{MR719666}
Tomek Bartoszy\'nski, \emph{Additivity of measure implies additivity of
  category}, Trans. Amer. Math. Soc. \textbf{281} (1984), no.~1, 209--213.
  \MR{719666}

\bibitem{MR1350295}
Tomek Bartoszy\'nski and Haim Judah, \emph{Set theory}, A K Peters, Ltd.,
  Wellesley, MA, 1995, On the structure of the real line. \MR{1350295}

\bibitem{BHRZ}
Zolt\'{a}n Buczolich, Bruce~H. Hanson, Martin Rmoutil, and Thomas Zuercher,
  \emph{On sets where $lip f$ is finite}.

\bibitem{MR1844397}
G.~A. Edgar, \emph{Packing measure in general metric space}, Real Anal.
  Exchange \textbf{26} (2000/01), no.~2, 831--852. \MR{1844397 (2002e:28011)}

\bibitem{MR2288081}
\bysame, \emph{Centered densities and fractal measures}, New York J. Math.
  \textbf{13} (2007), 33--87 (electronic). \MR{2288081 (2008b:28006)}

\bibitem{MR1039321}
Ryszard Engelking, \emph{General topology}, second ed., Sigma Series in Pure
  Mathematics, vol.~6, Heldermann Verlag, Berlin, 1989, Translated from the
  Polish by the author. \MR{1039321}

\bibitem{MR2118797}
Kenneth~J. Falconer, \emph{Fractal geometry}, second ed., John Wiley \& Sons
  Inc., Hoboken, NJ, 2003, Mathematical foundations and applications.
  \MR{2118797 (2006b:28001)}

\bibitem{fremlin5}
D.~H. Fremlin, \emph{Measure theory. {V}ol. 5, set-theoretic measure theory},
  2008, http://www.essex.ac.uk/maths/people/fremlin/mt5.2008/mt5.2008.tar.gz.

\bibitem{MR3511937}
Bruce~H. Hanson, \emph{Sets of non-differentiability for functions with finite
  lower scaled oscillation}, Real Anal. Exchange \textbf{41} (2016), no.~1,
  87--99. \MR{3511937}

\bibitem{HPZZ}
Bruce~H. Hanson, Pamela Pierce, Miroslav Zelen\'{y}, and Ond\v{r}ej Zindulka,
  \emph{Dimension inequality for lower lipschitz mappings}, in preparation.

\bibitem{MR1346667}
H.~Joyce and D.~Preiss, \emph{On the existence of subsets of finite positive
  packing measure}, Mathematika \textbf{42} (1995), no.~1, 15--24. \MR{1346667
  (96g:28010)}

\bibitem{MR1321597}
Alexander~S. Kechris, \emph{Classical descriptive set theory}, Graduate Texts
  in Mathematics, vol. 156, Springer-Verlag, New York, 1995. \MR{1321597}

\bibitem{MR1313694}
Tam{\'a}s Keleti, \emph{A peculiar set in the plane constructed by {V}itu\v
  skin, {I}vanov and {M}elnikov}, Real Anal. Exchange \textbf{20} (1994/95),
  no.~1, 291--312. \MR{1313694}

\bibitem{MR3159074}
Tam{\'a}s Keleti, Andr{\'a}s M{\'a}th{\'e}, and Ond{\v{r}}ej Zindulka,
  \emph{Hausdorff dimension of metric spaces and {L}ipschitz maps onto cubes},
  Int. Math. Res. Not. IMRN (2014), no.~2, 289--302. \MR{3159074}

\bibitem{MR1333890}
Pertti Mattila, \emph{Geometry of sets and measures in {E}uclidean spaces},
  Cambridge Studies in Advanced Mathematics, vol.~44, Cambridge University
  Press, Cambridge, 1995, Fractals and rectifiability. \MR{1333890 (96h:28006)}

\bibitem{MR3032324}
Manor Mendel and Assaf Naor, \emph{Ultrametric subsets with large {H}ausdorff
  dimension}, Invent. Math. \textbf{192} (2013), no.~1, 1--54. \MR{3032324}

\bibitem{MR560132}
Stephen~C. Milne, \emph{Peano curves and smoothness of functions}, Adv. in
  Math. \textbf{35} (1980), no.~2, 129--157. \MR{560132}

\bibitem{MR2979649}
Ale{\v{s}} Nekvinda and Ond{\v{r}}ej Zindulka, \emph{Monotone metric spaces},
  Order \textbf{29} (2012), no.~3, 545--558. \MR{2979649}

\bibitem{MR0154965}
A.~G. Vitushkin, L.~D. Ivanov, and M.~S. Melnikov, \emph{Incommensurability of
  the minimal linear measure with the length of a set}, Dokl. Akad. Nauk SSSR
  \textbf{151} (1963), 1256--1259. \MR{0154965 (27 \#4908)}

\bibitem{MR2957686}
Ond{\v{r}}ej Zindulka, \emph{Universal measure zero, large {H}ausdorff
  dimension, and nearly {L}ipschitz maps}, Fund. Math. \textbf{218} (2012),
  no.~2, 95--119. \MR{2957686}

\end{thebibliography}
\providecommand{\bysame}{\leavevmode\hbox to3em{\hrulefill}\thinspace}
\providecommand{\MR}{\relax\ifhmode\unskip\space\fi MR }
\providecommand{\MRhref}[2]{%
  \href{http://www.ams.org/mathscinet-getitem?mr=#1}{#2}
}
\providecommand{\href}[2]{#2}

\end{document}